\documentclass[11pt]{amsart}

\usepackage[utf8]{inputenc}
\usepackage{enumerate}
\usepackage{amsrefs}
\usepackage{dsfont}
\usepackage{amsmath}
\usepackage{amsthm}
\usepackage{amssymb}
\usepackage{amsfonts}
\usepackage{mathrsfs}  
\usepackage[shortlabels]{enumitem} 
\usepackage{color}
\usepackage{xcolor}
\usepackage[bookmarks=true,hyperindex,pdftex,colorlinks,citecolor=blue]{hyperref}
\usepackage{marginnote} 
\usepackage{graphicx} 
\usepackage{caption} 
\usepackage{soul} 
\usepackage{bm}
\usepackage[centertags]{amsmath}
\usepackage{tocvsec2}

\newcommand{\nn}{\mathbb{N}}

\newcommand{\co}{\text{cof}}

\newcommand{\eps}{\varepsilon}

\newcommand{\Ndb}{\mathbb N}

\newcommand{\Rdb}{\mathbb R}

\newcommand{\cal}{\mathcal}

\newcommand{\supp}{\text{supp}\,}

\newcommand{\spa}{\text{span}\,}

\theoremstyle{plain}
\newtheorem{theorem}{Theorem}[section]
\newtheorem{lemma}[theorem]{Lemma}
\newtheorem{corollary}[theorem]{Corollary}
\newtheorem{proposition}[theorem]{Proposition}

\newtheorem{thmAlfa}{Theorem}

\theoremstyle{definition}
\newtheorem*{definition*}{Definition}
\newtheorem{definition}[theorem]{Definition}

\theoremstyle{remark}
\newtheorem{remark}[theorem]{Remark}

\begin{document}
	
	
	\title[]{Universality, complexity and asymptotically uniformly smooth Banach spaces}
	
	\author{R.M.~Causey}
	\address{R.M~Causey}
	\email{rmcausey1701@gmail.com}

	\author{G.~Lancien}
	\address{G.~Lancien, Laboratoire de Math\'ematiques de Besan\c con, Universit\'e Bourgogne Franche-Comt\'e, 16 route de Gray, 25030 Besan\c con C\'edex, Besan\c con, France}
	\email{gilles.lancien@univ-fcomte.fr}
	
	\subjclass[2020]{}
	
	\begin{abstract}  For $1<p\le \infty$, we show the existence of a Banach space which is both injectively and surjectively universal for the class of all separable Banach spaces with an equivalent $p$-asymptotically uniformly smooth norm. We prove that this class is analytic complete in the class of separable Banach spaces. These results extend previous works by Kalton, Werner and Kurka in the case $p=\infty$. 
	\end{abstract}
	
	\maketitle
	
	\setcounter{tocdepth}{1}

	\section{Introduction}
	
	The notion of asymptotic smoothness has become very important in the recent developments of the geometry of Banach spaces. In this paper we study the existence of universal spaces and the complexity, in the sense of descriptive set theory, of a special class of asymptotically uniformly smooth Banach spaces. For a given $p \in (1,\infty]$, we are interested in the class of separable Banach spaces that admit an equivalent $p$-asymptotically uniformly smooth norm (we abbreviate $p$-AUS). In fact, it is one of the main results from \cite{CauseyPositivity2018} that the existence of an equivalent $p$-AUS norm is characterized by some $\ell_p$ upper estimates for weakly null trees or in terms of a two player infinite asymptotic game. We denote this property $\textsf{T}_p$ as a reference to the $\ell_p$ tree upper estimates. In Section \ref{definitions} we recall these definitions and the main characterizations of $\textsf{T}_p$.
	
	If we denote $\textsf{Sep}$ the class of all separable Banach spaces, it is known that $\textsf{T}_\infty \cap \textsf{Sep}$  coincides with the class of all Banach spaces that are isomorphic to a subspace of $c_0$. This result was first  explicitly proved in  \cite{GKL}, but could be readily extracted from ideas of N. Kalton and D. Werner in \cite{KW}. We also refer to \cite{JLPS} for the optimal proof. One of the goals of this paper is to find a space in $\textsf{T}_p \cap \textsf{Sep}$ which is both injectively and surjectively universal for the class $\textsf{T}_p \cap \textsf{Sep}$, with $1<p\le\infty$. This is achieved in Section \ref{universal}, where we prove.
	
	\begin{thmAlfa} For any $p\in (1,\infty]$, there exists a space $U_p$ in $\textsf{\emph{T}}_p\cap \textsf{\emph{Sep}}$ such that every Banach space in $\textsf{\emph{T}}_p \cap  \textsf{\emph{Sep}}$ is both isomorphic to a subspace and to a quotient of $U_p$. 
	\end{thmAlfa}
	We explain in Section \ref{universal} how this result can be deduced from previous works by Odell and Schlumprecht \cite{OS} and Freeman, Odell, Schlumprecht and Zs\'ak \cite{FOSZ}. However, we have chosen to explain the construction through the use of ``press down'' norms associated with spaces with finite dimensional decompositions. All the necessary background on press down norms is introduced in Section \ref{pressdown}.
	
	Finally, in Section \ref{descriptive}, we address the question of the topological complexity of the class $\textsf{T}_p \cap \textsf{Sep}$. We first recall the framework built by B. Bossard in \cite{Bossard2002} in which this problem can be rigorously formulated. Again, the question was only open for $p\in (1,\infty)$, as O. Kurka proved in \cite{Kurka2018} that the class of all Banach spaces that linearly embed into $c_0$ is analytic complete. We extend this result by showing the following.
	
	\begin{thmAlfa} For any $p\in (1,\infty]$, the class  $\textsf{\emph{T}}_p\cap \textsf{\emph{Sep}}$ is analytic complete. In particular it is not Borel.
	\end{thmAlfa}
	
	\section{Main characterizations of $p$-AUS-able spaces}\label{definitions}
	
	All Banach spaces are over the field $\mathbb{K}$, which is either $\mathbb{R}$ or $\mathbb{C}$.   We denote $B_X$ (resp. $S_X$) the closed unit ball (resp. sphere) of a Banach space $X$.    By \emph{subspace}, we shall always mean closed subspace.   Unless otherwise specified, all spaces are assumed to be infinite dimensional. 
	
	We start with the definition of the  \emph{modulus of asymptotic uniform smoothness} of $X$. If $X$ is infinite dimensional, for $\tau \geqslant 0$,  we define 
	\[\overline{\rho}_X(\tau) = \underset{y\in B_X}{\sup}\ \underset{E\in \co(X)}{\inf}\  \underset{x\in B_E}{\sup} \|y+\tau x\|-1,\] 
	where $\co(X)$ denotes the set of finite codimensional subspaces of $X$. For the sake of completeness, we define $\overline{\rho}_X(\tau)=0$ for all $\tau \geqslant 0$, when $X$ is finite dimensional.    We note that 
	\[\overline{\rho}_X(\tau) = \sup\{\underset{\lambda}{\lim\sup}\|y+\tau x_\lambda\|-1:(x_\lambda)\subset B_X\text{\ is a weakly null net}\}.\]  
	We say $X$ is \emph{asymptotically uniformly smooth} (in short \emph{AUS}) if 
	$$\lim_{\tau \to 0^+} \frac{\overline{\rho}_X(\tau)}{\tau}= 0.$$   
	We say $X$ is \emph{asymptotically uniformly smoothable} (\emph{AUS-able}) if $X$ admits an equivalent AUS norm. For $p\in (1,\infty)$, the norm $\|\ \|$ is said to be
	{\it $p$-asymptotically uniformly smooth} (in short $p$-AUS) if there exists $C>0$ such that
	$$\forall \tau\in [0,\infty),\ \overline{\rho}_X(\tau)\le C\tau^p.$$
	We say $X$ is $p$-\emph{asymptotically uniformly smoothable} ($p$-\emph{AUS-able}) if $X$ admits an equivalent $p$-AUS norm.  We say $X$ is \emph{asymptotically uniformly flat} (\emph{AUF}) if there exists $\tau_0>0$ such that $\overline{\rho}_X(\tau_0)=0$.  We say $X$ is \emph{asymptotically uniformly flattenable} (\emph{AUF-able}) if $X$ admits an equivalent AUF norm. Of course, $p$-AUS spaces and AUF spaces are AUS spaces. 
	
	The dual notion is provided by the following modulus defined on $X^*$ by
	$$ \overline{\delta}_X^*(\tau)=\inf_{x^*\in S_{X^*}}\sup_{E}\inf_{y^*\in S_E}\{\|x^*+\tau y^*\|-1\},$$
	where $E$ runs through $\co^*(X^*)$, the set of all weak$^*$-closed subspaces of
	$X^*$ of finite codimension. The norm of $X^*$ is said to be {\it weak$^*$  asymptotically uniformly convex} (in short AUC$^*$) if
	$$\forall \tau>0\ \ \ \ \overline{\delta}_X^*(\tau)>0.$$
	For $q\in [1,\infty)$, it is said to be
	{\it $q$-weak$^*$ asymptotically uniformly convex} (in short $q$-AUC$^*$) if there exists $c>0$ such that
	$$\forall \tau\in [0,1),\ \overline{\delta}_X^*(\tau)\ge c\tau^q.$$
	It is well known that the dual Young function of $\overline{\rho}_X$ is equivalent to  $\overline{\delta}^*_X$. Here is the precise version of this (Proposition 2.1 in \cite{DKLR2017}).
	
	\begin{proposition}\label{moduliduality} 
		There exists a universal constant $C\ge 1$ such that for any Banach space $X$ and any $0<\sigma,\tau<1$,
		\begin{enumerate}
			\item If $\overline{\rho}_X(\sigma)<\sigma \tau$, then $\overline{\delta}^*_X(C\tau)\ge \sigma\tau$.
			\item If $\overline{\delta}^*_X(\tau)>\sigma \tau$, then $\overline{\rho}_X(\frac{\sigma}{C}) \le \sigma\tau$.
			\item Let $p\in (1,\infty]$ and $q$ be its conjugate exponent. Then $\|\ \|_X$ is $p$-AUS if and only if $\|\ \|_{X^*}$ is $q$-AUC$^*$.
		\end{enumerate}
		
	\end{proposition}
	
	We now define the fundamental property, which turns out to characterize $p$-AUS renormability,  through a two-players game on a Banach space $X$. Fix $1<p\leqslant \infty$ and let $q$ be its conjugate exponent. We denote by $W_X$ the set of weak neighborhoods of $0$ in $X$. For $c>0$, we define the $T(c,p)$ game on $X$.  In the $T(c,p)$ game, Player I chooses a weak neighborhood $U_1\in W_X$, and Player II chooses $x_1\in U_1\cap B_X$. Player I chooses $U_2\in W_X$, and Player II chooses $x_2\in U_2\cap B_X$.  Play continues in this way until $(x_i)_{i=1}^\infty$ has been chosen. Player I wins if $\|(x_i)_{i=1}^\infty\|_q^w \leqslant c$, and Player II wins otherwise, where
	$$\|(x_i)_{i=1}^\infty\|_q^w = \inf\big\{c\in (0,\infty],\ \forall a=(a_i)_{i=1}^\infty \in \ell_p\ \|\sum_{i=1}^\infty a_ix_i\|\le c\|a\|_p\big\}.$$
	It is known (see \cite{CauseyPositivity2018}, section 3) that this game is determined. That is, either Player I or Player II has a winning strategy.  We now let $\textsf{t}_p(X)$ denote the infimum of those $c>0$ such that Player I has a winning strategy in the $T(c,p)$ game, provided such a $c$ exists, and we let $\textsf{t}_p(X)=\infty$ otherwise. We can now define our class $\textsf{T}_p$.
	
	\begin{definition} Let $p\in (1,\infty]$. We denote $\textsf{T}_p$ the class of all Banach spaces $X$ such that $\textsf{t}_p(X)<\infty$. 
	\end{definition}

	Let $D$ be a set. We denote $\varnothing$ the empty sequence and $D^{\leqslant n}=\{\varnothing\}\cup \cup_{i=1}^n D^i$, $D^{<\omega}=\cup_{n=1}^\infty D^{\le n}$,  $D^\omega$ the set of all infinite sequences whose members lie in $D$ and $D^{\leqslant \omega}=D^{<\omega}\cup D^\omega$.  For $s,t\in D^{<\omega}$, we let $s\smallfrown t$ denote the concatenation of $s$ with $t$. We let $|t|$ denote the length of $t$.  For $0\leqslant i\leqslant|t|$, we let $t|_i$ denote the initial segment of $t$ having length $i$, where $t|_0=\varnothing$.  If $s\in D^{<\omega}$, we let $s\prec t$ denote the relation that $s$ is a proper initial segment of $t$.
	
	Given $D$ a weak neighborhood basis of $0$ in $X$ and $(x_t)_{t\in D^{<\omega}}\subset X$, we say $(x_t)_{t\in D^{<\omega}}$ is  \emph{weakly null}   provided that for each $t\in \{\varnothing\}\cup D^{<\omega}$,\\  $(x_{t\smallfrown (U)})_{U\in D}$ is a weakly null net. Here $D$ is directed by reverse inclusion. 
	
	We can now recall the main characterizations of this class, which are proved in full generality in \cite{CauseyPositivity2018} and can be extracted from the proof of Theorem 3 of \cite{OS} in the separable case. 
	
	\begin{theorem}\label{ttheorem} Fix $1<p\leqslant \infty$ and let $q$ be conjugate to $p$. Let $X$ be a Banach space. The following are equivalent  
		\begin{enumerate}[(i)]
			\item $X\in \textsf{\emph{T}}_p$. 
			\item There exists a constant $c>0$ such that for any weak neighborhood basis $D$ at $0$ in $X$ and any weakly null $(x_t)_{t\in D^{<\omega}}\subset B_X$, there exists $\tau\in D^\omega$ such that $\|(x_{\tau|_i})_{i=1}^\infty\|_q^w \leqslant c$. 
			\item $X$ is $p$-AUS-able (resp. AUF-able if $p=\infty$). 
			\item $X$ admits an equivalent norm whose dual is $q$-AUC$^*$.
		\end{enumerate}
		Moreover, if $X\in \textsf{\emph{T}}_p$, then any separable subspace of $X$ has a separable dual.
	\end{theorem}

	\section{Press down and lift up norms}\label{pressdown}
	
	We introduce in this section two constructions of norms on a Banach space with a finite dimensional decomposition (we abbreviate FDD). We will recall the necessary background on FDD's, shrinking FDD's or boundedly complete FDD's. However, we refer the reader to \cite{LindenstraussTzafriri} for possibly missing information. One can trace back the use of lift up norms in the works of S. Prus \cite{Prus}. As we will see, it has also been widely exploited in the work of Odell and Schlumprecht \cite{OS} or Freeman, Odell, Schlumprecht and Zs\'{a}k \cite {FOSZ}. For the press down norm, we refer for instance to \cite{CauseyNavoyan} for a general presentation.

	We let $(u_i)_{i=1}^\infty$  denote the canonical basis of $c_{00}$ (the space of finitely supported scalar sequences). For subsets $E,F$ of $\nn$, we let $E<F$ denote the relation that $E=\varnothing$, $F=\varnothing$, or $\max E<\min F$.   For $f=\sum_{i=1}^\infty a_iu_i\in c_{00}$, we let $\supp(f)=\{i\in\nn: a_i\neq 0\}$.   For $f,g\in c_{00}$, we let $f<g$ denote the relation that $\supp(f)<\supp(g)$.  For $n\in\nn$, $E\subset \nn$, and $f\in c_{00}$, we let $n\leqslant E$ denote the relation that either $E=\varnothing$ or $n\leqslant \min E$, and we let $n\leqslant f$ denote the relation that $n\leqslant \supp(f)$.

	We recall that a \emph{finite dimensional decomposition} for a Banach space $Z$ is a sequence $\textsf{H}=(H_n)_{n=1}^\infty$ of finite dimensional, non-zero subspaces of $Z$ such that for any $z \in Z$, there exists a unique sequence $(z_n)_{n=1}^\infty \in \prod_{n=1}^\infty H_n$ such that $z=\sum_{n=1}^\infty z_n$. Then, we let $P^\textsf{H}_n$ denote the canonical projections $P^\textsf{H}_n(z) = z_n$, where $z=\sum_{n=1}^\infty z_n$ and $(z_n)_{n=1}^\infty \in \prod_{n=1}^\infty H_n$. For a finite or cofinite subset $I$ of $\nn$, we let $P^\textsf{H}_I=I^\textsf{H}=\sum_{n\in I}P^\textsf{H}_n$.  When no confusion can arise, we omit the superscript and simply denote $I^\textsf{H}$ by $I$. It follows from the principle of uniform boundedness that $\sup\{\|I^\textsf{H}\|,\ I\subset \Ndb\ \text{is an interval}\}$ is finite. We refer to this quantity as the \emph{projection constant} of $\textsf{H}$  in $Z$. If the projection constant of $\textsf{H}$  in $Z$ is $1$, we
	say $\textsf{H}$ is \emph{bimonotone}. It is well-known that if $\textsf{H}$ is an FDD for $Z$, then there exists an equivalent norm $|\ |$ on $Z$ such that $\textsf{H}$ is a bimonotone FDD of $(Z,|\ |)$. Finally, we denote $c_{00}(\textsf{H})$ the space of finite linear combinations of elements in $H_1,\ldots,H_n,\ldots$

	Assume now that $Z$ is a Banach space with FDD $(H_j)_{j=1}^\infty=\textsf{{H}}$ and $p\in [1,\infty]$. We first define the \emph{lift up} norm associated with $Z,\textsf{{H}}$ and $p$ and the corresponding Banach space $Z_{\vee,p}(\textsf{{H}})$ as the completion of $c_{00}(\textsf{H})$ under the norm
	$$\|z\|_{\vee,p} = \sup\Bigl\{\Big\|\big(\|J_iz\|_Z\big)_{i=1}^\infty\Big\|_{\ell_p}: J_i\neq \varnothing, J_1<J_2<\ldots, \nn=\cup_{i=1}^\infty J_i\Bigr\}.$$
	
	We now define the \emph{press down} norm associated with $Z,\textsf{{H}}$ and $p$ and the corresponding Banach space $Z_{\wedge,p}(\textsf{{H}})$. For $z\in c_{00}(\textsf{H})$, we set 
	$$[z]_{\wedge,p} = \inf\Bigl\{\Big\|\big(\|J_iz\|_Z\big)_{i=1}^\infty\Big\|_{\ell_p}: J_i\neq \varnothing, J_1<J_2<\ldots, \nn=\cup_{i=1}^\infty J_i\Bigr\}.$$
	Note that the conditions imposed on the $J_i$'s imply that they are intervals. Then we define
	$$\|z\|_{\wedge,p} = \inf\Big\{\sum_{n=1}^m [z_n]_{\wedge,p}:\ m\in \Ndb, z_n\in c_{00}(\textsf{H}), z=\sum_{n=1}^mz_n\Big\}.$$
	It is easily checked that $\|\ \|_{\wedge,p}$ is a norm and we denote $Z_{\wedge,p}(\textsf{{H}})$ the completion of $c_{00}(\textsf{H})$ under $\|\ \|_{\wedge,p}$. It is clear that $\textsf{H}$ is a FDD for $Z_{\wedge,p}(\textsf{{H}})$ and that it is bimonotone if $\textsf{H}$ is a bimonotone FDD of $Z$. 	We shall use the following easy technical simplification.
	
	\begin{proposition}\label{stubs} Let $Z$ be a Banach space with FDD $\textsf{\emph{H}}$ and $p\in (1,\infty]$. Then for any $l\le m$ in $\Ndb$ and any  $z\in \oplus_{j=l}^m H_j$, there exist intervals $I_1<\ldots <I_n$ such that $\cup_{i=1}^n I_i=[l,m]$ and 
		$$[z]_{\wedge,p} = \Big\|\big(\|I_iz\|_Z\big)_{i=1}^n\Big\|_{\ell_p}.$$   
	\end{proposition}
	
	The following lemma will be useful for our estimates. 
	\begin{lemma}\label{pressdown_conv} Assume moreover that $\textsf{\emph{H}}$ is a bimonotone FDD of $Z$. Let $I_1<I_2$ intervals of $\Ndb$ and $z_1,z_2 \in c_{00}(\textsf{\emph{H}})$ with $\supp(z_i)\subset I_i$, then  
		$$z_1+z_2 \in \overline{\text{conv}}\big\{y_1+y_2,\ \supp(y_i)\subset I_i,\ [y_i]_{\wedge,p}\le \|z_i\|_{\wedge,p}\big\}.$$
	\end{lemma}
	\begin{proof} It follows from standard arguments that the closed unit ball of $Z_{\wedge,p}(\textsf{{H}})$ is the closed convex hull of those $z\in c_{00}(\textsf{\emph{H}})$ such that $[z]_{\wedge,p}\le 1$. Since $\textsf{{H}}$ is bimonotone, we clearly have that for any interval $I\subset \Ndb$ and $z\in Z$, $[Iz]_{\wedge,p}\le [z]_{\wedge,p}$. Combining these two facts, we get that any $z$ in the closed linear span of $\{H_j, j\in I\}$ is in the closed convex hull of those $y$'s in the closed linear span of $\{H_j, j\in I\}$ satisfying $[y]_{\wedge,p}\le \|z\|_{\wedge,p}$. This finishes the proof. 
	\end{proof}
	
	We now detail a simple but crucial effect of the press down procedure. 
	
	\begin{proposition}\label{resid1} Let $Z$ be a Banach space with bimonotone FDD $\textsf{\emph{H}}$. Fix $1<p\leqslant \infty$. Then, the FDD $\textsf{\emph{H}}$ of the space $Z_{\wedge,p}(\textsf{\emph{H}})$  satisfies upper $p$ block estimates, with constant $1$. More precisely, for all $x,y \in c_{00}(\textsf{\emph{H}})$ so that $x<y$, we have 
$$\|x+y\|_{\wedge,p}^p \le \|x\|_{\wedge,p}^p+\|y\|_{\wedge,p}^p, \ \text{if}\ p\in (1,\infty)$$
and 
$$\|x+y\|_{\wedge,\infty} = \max\{\|x\|_{\wedge,\infty},\|y\|_{\wedge,\infty}\}, \ \text{if}\ p=\infty.$$
In particular $Z_{\wedge,p}(\textsf{\emph{H}})$ has $\textsf{\emph{T}}_p$. 
	\end{proposition}
	
	\begin{proof} We only detail the case $p\in (1,\infty)$. Let $x,y \in c_{00}(\textsf{{H}})$ so that $x<y$ and $I<J$ intervals such that $\supp(x)\subset I$, $\supp(y) \subset J$ and $I\cup J$ is an interval containing $1$. By Proposition \ref{stubs}, there exist $m,n\in\nn$ and $I_1<\ldots<I_{m+n}$ such that $\cup_{i=1}^m I_i=I$, $\cup_{i=m+1}^{m+n} I_i=J$ and 
		$$[x]_{\wedge,p}^p = \sum_{i=1}^m \|I_ix\|_Z^p\ \ \text{and}\ \  [y]_{\wedge,p}^p = \sum_{i=m+1}^{m+n} \|I_iy\|_Z^p.$$ 
		We complete with $(I_j)_{j>m+n}$ an arbitrary interval partition of $\nn\setminus (I\cup J)$.  Then 
		\begin{align*}
			\|x+y\|_{\wedge,p}^p \leqslant [x+y]_{\wedge,p}^p \leqslant \sum_{i=1}^\infty \|I_i(x+y)\|_Z^p &= \sum_{i=1}^m \|I_ix\|_Z^p+\sum_{i=m+1}^\infty \|I_iy\|_Z^p\\ 
			&= [x]_{\wedge,p}^p+[y]_{\wedge,p}^p.
		\end{align*}
		The conclusion now follows from Lemma \ref{pressdown_conv}.\\
		These upper $p$ block estimates, with constant $1$ clearly imply that $\|\ \|_{\wedge,p}$ is $p$-AUS and therefore that $Z_{\wedge,p}(\textsf{H}) \in \textsf{{T}}_p$.	\end{proof}

	We now need to recall some basics on dual FDD's. If $Z$ is a Banach space with FDD $\textsf{H}=(H_n)_{n=1}^\infty$, we let $\textsf{H}^*$ denote the sequence $(H^*_n)_{n=1}^\infty$. Here, $H^*_n$ is identified with the sequence $((P^\textsf{H}_n)^*(Z^*))_{n=1}^\infty$. This identification does not need to be isometric if $\textsf{H}$ is not bimonotone in $Z$.  We let $Z^{(*)}=\overline{c_{00}(\textsf{H}^*)} \subset Z^*$. The FDD $\textsf{H}$ is said to be \emph{shrinking} if $Z^{(*)}=Z^*$, which occurs if and only if any bounded block sequence with respect to $\textsf{H}$ is weakly null. The FDD $\textsf{H}$ is said to be \emph{boundedly complete} if $\textsf{H}^*$ is a shrinking FDD of $Z^{(*)}$ (in that case $Z$ is canonically isomorphic to $(Z^{(*)})^*)$.
	
	The duality between press down and lift up norms is described by the following proposition, which is a special case of Proposition 2.1 in \cite{CauseyNavoyan}.
	
	\begin{proposition}\label{dualitydownup} Let $Z$ be a Banach space with bimonotone shrinking FDD $\textsf{\emph{H}}$. Let $p \in (1,\infty]$ and $q$ be its conjugate exponent. Then $Z_{\wedge,p}(\textsf{\emph{H}})^*$ is canonically isometric to $Z_{\vee,q}(\textsf{\emph{H}}^*)$. 
	\end{proposition}
	
	The following is an immediate corollary of Proposition \ref{resid1}. 
	
	\begin{corollary}   Let $Z$ be a Banach space with bimonotone FDD $\textsf{\emph{H}}$ and $p\in (1,\infty]$, then $\textsf{\emph{H}}$ is a shrinking FDD of $Z_{\wedge,p}(\textsf{\emph{H}})$. 
	\end{corollary}
	
	\begin{remark} Note that for $p=\infty$, it follows from Proposition \ref{resid1} that if $Z$ is a Banach space with bimonotone FDD $\textsf{{H}}$, then  $Z_{\wedge,\infty}(\textsf{{H}})$ is simply isometric to the $(\sum_{i=1}^\infty H_i)_{c_0}$. In particular, for any $\eps>0$, it is $(1+\eps)$-isomorphic to a subspace of $c_0$.
	\end{remark}

	\section{A universal space for separable $p$-AUS-able spaces}\label{universal}
	
	\subsection{The main result}	
	The statements from this section can all be found either explicitly or implicitly in \cite{FOSZ}. We recall the main steps. The only difference with \cite{FOSZ} is the emphasis that we choose to put on press down instead of lift up norms.
	
	\begin{theorem}\label{FDD} Let $p\in (1,\infty]$ and $X$ be a separable Banach space with $\textsf{\emph{T}}_p$. Then 
		\begin{enumerate}[(i)]
			\item	There exists a Banach space $Z$ with a bimonotone shrinking FDD $\textsf{\emph{H}}$ such that $X$ is isomorphic to a subspace of $Z_{\wedge,p}(\textsf{\emph{H}})$. 
			\item There exists a Banach space $Y$ with a bimonotone shrinking FDD $\textsf{\emph{N}}$ such that $X$ is isomorphic to a quotient of $Y_{\wedge,p}(\textsf{\emph{N}})$.
		\end{enumerate}
	\end{theorem}
	
	\begin{proof} Assume that $X \in \textsf{T}_p\cap \textsf{Sep}$. Then, by Theorem \ref{ttheorem}, $X^*$ is separable and $X$ satisfies $\ell_p$ upper tree estimates (condition $(ii)$ in Theorem \ref{ttheorem}). This easily implies that $X$ satisfies subsequential $\ell_p$ upper tree estimates in the sense of \cite{FOSZ}, from which we just have to apply Theorem 1.1 in \cite{FOSZ} to conclude. To be more accurate, this information is in the proof of this Theorem 1.1, where the authors actually show the existence of weak$^*$ isomorphisms  from $X^*$ onto a quotient of $Z^*_{\vee,q}(\textsf{{H}}^*)$ by a weak$^*$ closed subspace (resp. onto a weak$^*$ closed subspace of $Y_{\vee,q}(\textsf{{N}}^*)$), where $q$ is the conjugate exponent of $p$. The conclusion then follows from Proposition \ref{dualitydownup}. 
	\end{proof}

	The next idea from \cite{FOSZ} is to use the complementably universal space for Banach spaces with an FDD built by Schechtman in \cite{Schechtman1975}. More precisely, Schechtman proved the following: for a given sequence $\textsf{J}=(J_n)_{n=1}^\infty$ of finite dimensional normed spaces which is dense in the space of all finite dimensional normed spaces for the Banach-Mazur distance,  there exists a Banach space $W(\textsf{J})$ with bimonotone FDD $\textsf{J}$ such that if $Z$ is any Banach space with bimonotone FDD $\textsf{H}$, then there exist integers $m_1<m_2<\ldots$ and a bounded, linear operator $A:Z\to W(\textsf{J})$ such that $A(H_n)=J_{m_n}$ for all $n\in \Ndb$  and 
	$$\forall z \in Z,\ \ \frac12\|z\|_Z \leqslant \|Az\|_{W(\textsf{J})} \leqslant 2\|z\|_Z$$ and such that $A(Z)=\overline{\text{span}}\{J_{m_n}:n\in\nn\}$ is $1$-complemented in $W(\textsf{J})$ via the map $P:w\mapsto \sum_{n=1}^\infty P^\textsf{J}_{m_n}w$. Then one can prove:

	\begin{proposition}\label{Pressdownuniversal} Let $p\in (1,\infty]$.\\ 
		Keeping the above notation, we have that $A:c_{00}(\textsf{\emph{H}})\to c_{00}(\textsf{\emph{J}})$ extends to an isomorphic embedding $\tilde{A}:Z_{\wedge,p}(\textsf{\emph{H}})\to W_{\wedge,p}(\textsf{\emph{J}})$, the range of which is still $1$-complemented in $W_{\wedge,p}(\textsf{\emph{J}})$ by $P$.
	\end{proposition}
	
	Combining Theorem \ref{FDD} and Proposition \ref{Pressdownuniversal}, we deduce the following ultimate result. We insist on the fact that it can be extracted from \cite{FOSZ}, although it is only stated there for embeddings (Corollary 3.3) and the construction is made in terms of lift up norms in the dual FDD's.
	
	\begin{theorem}\label{universal_t} Let $p\in (1,\infty]$ and $\textsf{\emph{J}}$ be a Banach-Mazur dense sequence of finite dimensional normed spaces. Then every space in $\textsf{\emph{T}}_p\cap \textsf{\emph{Sep}}$ is both isomorphic to a subspace and to a quotient of $W_{\wedge,p}(\textsf{\emph{J}})$. In other words, $W_{\wedge,p}(\textsf{\emph{J}})$  is both universal and surjectively universal for $\textsf{\emph{T}}_p\cap \textsf{\emph{Sep}}$. 
	\end{theorem}
	
	\begin{remark} We conclude this section with the following elementary and well known remark: one cannot expect an analogous statement for $p$-uniformly smoothable spaces. Indeed, for all $q\in [2,\infty)$, the space $\ell_q$ is $2$-uniformly smooth. Now let $X$ be a Banach space containing an isomorphic copy of all $2$-uniformly smooth  separable Banach spaces. Then $X$ isomorphically contains $\ell_q$ for $2\le q<\infty$. It follows from standard arguments that $X$ has trivial cotype, therefore is not super-reflexive and thus does not admit any equivalent uniformly smooth norm. 
	\end{remark}
	
	\subsection{A few remarks on this universal space}. 
	
	For a Banach space $X$, we denote by $\cal S(X)$ the class of all Banach spaces isomorphic to a subspace of $X$ and by $\cal Q(X)$ the class of all Banach spaces isomorphic to a quotient of $X$. In the following proposition, the case $p=\infty$ is an old result by W. B. Johnson and M. Zippin \cite{JohnsonZippin}.
	
	\begin{proposition} Let $p\in (1,\infty]$ and $\textsf{\emph{J}}$ be a Banach-Mazur dense sequence of finite dimensional normed spaces. Then 
		$$\textsf{\emph{T}}_p\cap \textsf{\emph{Sep}} = \cal S(W_{\wedge,p}(\textsf{\emph{J}}))=\cal Q(W_{\wedge,p}(\textsf{\emph{J}})).$$
	\end{proposition}
	
	\begin{proof} By Proposition \ref{resid1}, $W_{\wedge,p}(\textsf{{J}})$ has $\textsf{{T}}_p$. The property $\textsf{{T}}_p$ passes clearly to subspaces and also passes to quotients, by characterization $(iv)$ of Theorem \ref{ttheorem}. The other inclusions are insured by Theorem \ref{universal_t}.
		
		\begin{remark} Note that  $W_{\wedge,\infty}(\textsf{J})$ is not isomorphic to $c_0$. For instance, because $W_{\wedge,\infty}(\textsf{J})^*$ uniformly contains the $\ell_\infty^n$'s while $\ell_1$ does not. We recall in passing that $\textsf{{T}}_\infty\cap \textsf{{Sep}}=\cal S(c_0)$ and also that  $\cal Q(c_0) \subsetneq \cal S(c_0)$. See \cite{JohnsonZippin} for the inclusion $\cal Q(c_0) \subset \cal S(c_0)$ and note that the dual of a quotient of $c_0$ has cotype $2$ unlike the dual of  $Y=(\sum_{n=1}^\infty \ell_1^n)_{c_0}$ and $Y$ linearly embeds into $c_0$.
		\end{remark}

	\end{proof}
	
	\begin{proposition} Let $p\in (1,\infty]$ and $\textsf{\emph{J}}$ be a Banach-Mazur dense sequence of finite dimensional normed spaces. Then, the  space $\ell_p(W_{\wedge,p}(\textsf{\emph{J}}))$ ($c_0(W_{\wedge,\infty}(\textsf{\emph{J}}))$ if $p=\infty$) is isomorphic to  $W_{\wedge,p}(\textsf{\emph{J}})$. 
	\end{proposition}
	
	\begin{proof} We treat the case $p\in (1,\infty)$ (the case $p=\infty$ is identical). Denote $X=W_{\wedge,p}(\textsf{{J}})$. We first show that $Z=\ell_p(X)$ is isomorphic to a complemented subspace of $X$. Indeed, there exists a bijection $\Phi:\Ndb \to \Ndb \times \Ndb$ such that the space $Z$ has an FDD $\textsf{K}$ with the property that, $K_n=J_{k}$ if $\Phi(n)=(i,k)$ and $\|x+y\|_{Z}^p \le \|x\|_{Z}^p+\|y\|_{Z}^p$, whenever $x,y \in c_{00}(\textsf{{K}})$ with $x<y$. Clearly, this upper $\ell_p$ block estimate implies that $Z_{\wedge,p}(\textsf{{K}})$ is isomorphic to $Z$. It then follows from Proposition \ref{Pressdownuniversal} that $Z$ is isomorphic to a complemented subspace of $X$. 
		
		We now conclude the proof with Pe\l czynski's classical decomposition method. Let $E$ be a subspace of $X$ such that $X\simeq \ell_p(X)\oplus E$. Then 
		\begin{align*}
			\ell_p(X)&\simeq \ell_p(\ell_p(X)\oplus_p E)\simeq \ell_p(\ell_p(X))\oplus_p \ell_p(E) \simeq \ell_p(\ell_p(X))\oplus_p \ell_p(E)\oplus_p E\\ 
			&\simeq \ell_p(\ell_p(X)\oplus_p E)\oplus_p E \simeq \ell_p(X)\oplus_p E\simeq X.
		\end{align*}
	\end{proof}
	
	We conclude this section by showing that $W_{\wedge,p}(\textsf{{J}})$ does not depend (up to isomorphism) on the choice of the dense sequence $\textsf{{J}}$.
	
	\begin{proposition} Let $p\in (1,\infty]$ and $\textsf{\emph{J}}$ and $\textsf{\emph{K}}$ be two Banach-Mazur dense sequences of finite dimensional normed spaces. Then $W_{\wedge,p}(\textsf{\emph{J}})$ is isomorphic to $W_{\wedge,p}(\textsf{\emph{K}})$.
	\end{proposition}
	
	\begin{proof} Denote $X=W_{\wedge,p}(\textsf{{J}})$ and  $Y=W_{\wedge,p}(\textsf{{K}})$. As in the previous proof, we deduce from Proposition \ref{Pressdownuniversal} that $Y$ is isomorphic to a complemented subspace of $X$ and $X$ is isomorphic to a complemented subspace of $Y$. Let $E$ be a subspace of $X$ such that $X\simeq Y \oplus E$. We can now apply Pe\l czynski's decomposition method  together with the previous proposition to get:
		\begin{align*}
			X\oplus_p Y&\simeq \ell_p(Y\oplus_p E)\oplus_p Y \simeq \ell_p(Y)\oplus_p Y \oplus_p \ell_p(E) \simeq \ell_p(Y)\oplus_p \ell_p(E)\\
			&\simeq \ell_p(Y\oplus_p E) \simeq \ell_p(X) \simeq X.
		\end{align*}
		For symmetric reasons, $X\oplus_p Y \simeq Y$. This finishes the proof. 
	\end{proof}

	\section{Complexity of the class of separable $p$-AUS-able spaces}\label{descriptive}
	
	\subsection{Preliminaries}
	
	We recall the setting introduced by B. Bossard in \cite{Bossard2002} in order to apply the tools from descriptive set theory to the class $\textsf{Sep}$ of separable Banach spaces. We also refer the reader to the more recent paper by G. Godefroy and J. Saint-Raymond \cite{GSR}, where an even more complete topological frame is presented. 
	
	A \emph{Polish space} (resp. \emph{topology}) is a separable completely metrizable space (resp. topology). A set $X$ equipped with a $\sigma$-algebra is called a \emph{standard Borel space} if the $\sigma$-algebra is generated by a Polish topology on $X$. A subset of such a  standard Borel space $X$ is called \emph{Borel} if it is an element of the corresponding $\sigma$-algebra and it is called \emph{analytic} (or a $\Sigma_1^1$-set) if there exist a standard Borel space $Y$ and a Borel subset $B$ of $X \times Y$ such that $A$ is the projection of $B$ on the first coordinate. The complement of an analytic set is called a \emph{coanalytic} set (or a $\Pi_1^1$ set). A subset $A$ of a standard Borel space $X$ is called \emph{$\Sigma_1^1$-hard} if for every $\Sigma_1^1$ subset $B$ of a standard Borel space $Y$, there exists a Borel map $f:Y \to X$ such that $f^{-1}(A)=B$ and it is called \emph{$\Sigma_1^1$-complete} if it is both $\Sigma_1^1$ and $\Sigma_1^1$-hard. We refer the reader to the textbook \cite{Kechris} for a thorough exposition of descriptive set theory. 
	
	Let $X$ be a Polish space. Then, the set $\cal F(X)$ of all closed subsets of $X$ can be equipped with its \emph{Effros-Borel structure}, defined as the $\sigma$-algebra generated by the sets $\{F \in \cal F(X),\ F\cap U\neq \emptyset\}$, where $U$ varies over the open subsets of $X$. Equipped with this $\sigma$-algebra, $\cal F(X)$ is a standard Borel space. 
	
	Following Bossard \cite{Bossard2002}, we now introduce the fundamental coding of separable Banach spaces. It is well known that $C(\Delta)$, the space of scalar valued continuous functions on the Cantor space $\Delta=\{0,1\}^\Ndb$, equipped with the sup-norm, contains an isometric linear copy of every separable Banach space. We equip $\cal F(C(\Delta))$ with its corresponding Effros-Borel structure. Then, we denote 
	$$\textsf{SB}=\{F\in \cal F(C(\Delta)),\ F\ \text{is a linear subspace of}\ C(\Delta)\},$$
	considered as a subspace of  $\cal F(C(\Delta))$. Then $\textsf{SB}$ is a Borel subset of $\cal F(C(\Delta))$ (\cite{Bossard2002}, Proposition 2.2) and therefore a standard Borel space, that we call the \emph{standard Borel space of separable Banach spaces}. 
	
	Let us now denote $\simeq$ the isomorphism equivalence relation on $\textsf{SB}$. The \emph{fundamental coding of separable Banach spaces} is the quotient map $c:\textsf{SB} \to  \textsf{SB}/\simeq$. We can now give the following definition.
	
	\begin{definition} A family $\textsf{G}\subset \textsf{SB}/\simeq$ is Borel (resp. analytic, coanalytic) if $c^{-1}(\textsf{G})$  is Borel (resp. analytic, coanalytic) in $\textsf{SB}$. 
	\end{definition}
	
	It is worth noting that there are other natural codings of separable Banach spaces, for instance as quotients of $\ell_1$. They yield the same definition of Borel or analytic classes of separable Banach spaces, as it is shown in \cite{Bossard2002}.

	\subsection{The class  $\textsf{Sep} \cap\textsf{T}_p$ is analytic complete.}
	
	It is easily seen that the class of separable Banach spaces that linearly embed into a fixed separable Banach space $Z$ is $\Sigma_1^1$ (see Lemma 3.6. in \cite{Bossard1994} for instance). So, it follows immediately from the existence of a universal space for the class $\textsf{Sep} \cap \textsf{T}_p$ that $\textsf{Sep} \cap \textsf{T}_p$ is $\Sigma_1^1$. The main purpose of this section will be to show that this is optimal.
	
	\begin{theorem}\label{complete} Let $p \in (1,\infty]$. Then, the class  $\textsf{\emph{Sep}} \cap\textsf{\emph{T}}_p$ is $\Sigma_1^1$-complete.
	\end{theorem}
	
	The case $p=\infty$ was settled by O. Kurka in \cite{Kurka2018}, where he showed that the class of all Banach spaces that linearly embed into $c_0$ (which coincides with $\textsf{Sep} \cap\textsf{T}_\infty$) is $\Sigma_1^1$-complete. Our result will follow from an adaptation of Kurka's argument, which relies in particular on the use  of a Tsirelson-type Banach space constructed by S. A. Argyros and I. Deliyanni in \cite{ArgyrosDeliyanni1997} that we shall describe now. We denote $\cal K(2^\Ndb)$, the set of compact subsets of $2^\Ndb$, the power set of $\Ndb$. Then we equip $\cal K(2^\Ndb)$ with its Hausdorff topology. For $\cal M \in \cal K(2^\Ndb)$, one can define a space, denoted $\cal T s^*[\cal M,\frac12]$ in \cite{Kurka2018}, but that we will just  denote $T_{\cal M}$ here. Let us gather here the properties of $T_{\cal M}$ that we shall need (see \cite{Kurka2018}).
	
	\begin{proposition}\label{Kurka} Let $\cal M \in \cal K(2^\Ndb)$. Then
		\begin{enumerate}
			\item The canonical basis $(u_n)_{n=1}^\infty$ of $T_{\cal M}$ is normalized and $1$-unconditional.
			\item If $\cal M$ contains all $3$ elements sets, then $(u_n)_{n=1}^\infty$ is shrinking.
			\item If $\cal M$ consists only of finite sets, then $(u_n)_{n=1}^\infty$ is boundedly complete.
			\item If $\cal M$ contains an infinite set, then $T_{\cal M}$ is isomorphic to a $c_0$-sum of finite dimensional spaces. More precisely, there exists an infinite sequence $1=m_0<m_1<\ldots<m_k<\ldots$ such that for all $x=(x_n)\in c_{00}$:
			$$\sup_{k\in \Ndb}\Big\|\sum_{n=m_{k-1}}^{m_k-1} x_nu_n\Big\|_{T_{\cal M}} \le \Big\|\sum_{n=1}^\infty x_nu_n\Big\|_{T_{\cal M}} \le 2 \sup_{k\in \Ndb}\Big\|\sum_{n=m_{k-1}}^{m_k-1} x_nu_n\Big\|_{T_{\cal M}}.$$
		\end{enumerate}
	\end{proposition}

	A key tool in Kurka's argument is the following particular case of a theorem due to Hurewicz (see \cite{Kurka2018} and references therein).
	
	\begin{theorem}[Hurewicz]\label{Hurewicz} Let $\cal I$ be the set of all infinite subsets of $\Ndb$. Then $\{\cal M\in \cal K(2^\Ndb), \cal M \cap \cal I\neq \emptyset\}$ is $\Sigma_1^1$-complete.
	\end{theorem}

	We also need to recall the construction of the \emph{$q$-convexification} of a Banach space $X$. So let $q\in [1,\infty)$ and $X$ be a Banach space with a normalized $1$-unconditional basis $(e_n)_{n=1}^\infty$. Let 
	$$X^q=\Big\{x=(x_n)_{n=1}^\infty \in \Rdb^\Ndb,\ x^q=\sum_{n=1}^\infty |x_n|^qe_n \in X\Big\}$$
	and endow it with the norm $\|x\|_{X^q}=\|x^q\|_X^{1/q}$. We also denote $(e_n)_{n=1}^\infty$ the sequence of coordinate vectors in $X^q$. It is clear that $(e_n)_{n=1}^\infty$ is a  normalized $1$-unconditional basis of $X^q$ and that $X^1$ is isometric to $X$. Also, the triangle inequality implies that $X^q$ is $q$-convex with constant $1$, meaning that for any $x^1,\ldots,x^n \in X^q$
	$$\Big\|\sum_{j=1}^\infty\big(|x_j^1|^q+\cdots+|x_j^n|^q\big)^{1/q}e_j\Big\|_{X^q} \le \big(\|x^1\|_{X^q}^q+\cdots+\|x^n\|_{X^q}^q\big)^{1/q}.$$
	Note that it follows that if $x^1,\ldots,x^n \in X^q$ have disjoint supports with respect to $(e_n)_{n=1}^\infty$, then 
	$$\|x^1+\cdots+x^n\|_{X^q}^q\le \|x^1\|_{X^q}^q+\cdots+\|x^n\|_{X^q}^q.$$
	In particular, the norm of $X^q$ is clearly $q$-AUS and therefore  $X^q \in \textsf{T}_q$. Note also that, if $q>1$, the above inequality implies that $(e_n)$ is a shrinking basis of $X^q$.

	We shall need a few other properties of $X^q$ that we list in the next lemmas, in which we will always assume that $(e_n)_{n=1}^\infty$ is a normalized $1$-unconditional basis of the Banach space $X$ and $q\in [1,\infty)$.
	
	\begin{lemma}\label{L1} Assume moreover that there exist a constant $C\ge 1$ and an infinite sequence $1=m_0<m_1<\ldots<m_k<\ldots$ such that for all $x=(x_n)\in c_{00}$:
		$$\frac{1}{C}\sum_{k\in \Ndb}\Big\|\sum_{n=m_{k-1}}^{m_k-1} x_ne_n\Big\|_{X} \le \Big\|\sum_{n=1}^\infty x_ne_n\Big\|_{X} \le C \sum_{k\in \Ndb}\Big\|\sum_{n=m_{k-1}}^{m_k-1} x_ne_n\Big\|_{X}.$$
		Then $X^q$ is isomorphic to an $\ell_q$-sum of finite dimensional spaces.
	\end{lemma}
	
	\begin{proof} It is clear that for all $x=(x_n)\in c_{00}$,
		$$\frac{1}{C}\sum_{k\in \Ndb}\Big\|\sum_{n=m_{k-1}}^{m_k-1} x_ne_n\Big\|_{X^q}^q \le \Big\|\sum_{n=1}^\infty x_ne_n\Big\|_{X^q}^q \le C \sum_{k\in \Ndb}\Big\|\sum_{n=m_{k-1}}^{m_k-1} x_ne_n\Big\|_{X^q}^q.$$
		
	\end{proof}

	Assume now that $(e_n)_{n=1}^\infty$ is a boundedly complete basis of $X$. It is immediate to check that the coordinate vectors, still denoted $(e_n)_{n=1}^\infty$, form a boundedly complete basis of $X^q$. In this situation, we shall denote $(X^q)_*$ the predual of $X^q$ given by the closed linear span in $(X^q)^*$ of the dual basis of $(e_n)_{n=1}^\infty$. Recall that, if $q>1$, $(e_n)$ is also a shrinking basis of $X^q$, which is therefore reflexive. We can now prove the following.

	\begin{lemma}\label{L3} Let $q\in (1,\infty)$ and $p$ be its conjugate exponent. Assume that  $X$ is reflexive. Then $\ell_q$ is not isomorphic to any subspace of $X^q$. Moreover $(X^q)_*$ does not belong to $\textsf{\emph{T}}_p$. 
	\end{lemma}
	
	\begin{proof} Assume that $\ell_q$ is isomorphic to a subspace of $X^q$. Then we can find a normalized sequence $(x_n)$ in $X^q$ which is equivalent to the canonical basis of $\ell_q$. Since this sequence is weakly null in $X^q$, we can as well assume that the $x_n$'s have finite consecutive disjoint supports with respect to the basis $(e_n)$. Denote now $y_n=x_n^q \in X$, which is normalized in $X$. Using the $1$-unconditionality of $(e_n)$ in $X$, it follows from our assumptions that there exists $C>0$ such that for any $a_1,\ldots,a_n \in \Rdb$:
		$$\Big\|\sum_{i=1}^n a_i y_i\Big\|_X=\Big\|\Big(\sum_{i=1}^n |a_i|^{1/q} x_i\Big)^q\Big\|_{X}=\Big\|\sum_{i=1}^n |a_i|^{1/q} x_i\Big\|_{X^q}^q \ge C\sum_{i=1}^n |a_i|.$$
		This implies that $(y_n)$ is equivalent, for the norm of $X$, to the canonical basis of $\ell_1$, which contradicts the fact that $X$ is reflexive. We have proved that $\ell_q$ is not isomorphic to any subspace of $X^q$.
		
		Assume now that $(X^q)_* \in \textsf{T}_p$. We already explained that $X^q \in \textsf{T}_q$. Then we can inductively construct two normalized sequences $(x_n)$ in $X^q$ and $(x_n^*)$ in $(X^q)_*$ with consecutive supports with respect to the respective bases of $X^q$ and $(X^q)_*$, that are biorthogonal to each other and such that $(x_n)$ is dominated by the canonical basis of $\ell_q$ and $(x_n^*)$ is dominated by the canonical basis of $\ell_p$ (also use the fact the bases are $1$-unconditional). It then readily follows from H\"{o}lder's inequality that $(x_n)$ is equivalent to the canonical basis of $\ell_q$, which is a contradiction.  
	\end{proof}

	We are now ready to prove our result.
	
	\begin{proof}[Proof of Theorem \ref{complete}] Fix $p\in (1,\infty)$ and denote $q$ its conjugate exponent. 
		
		Let us denote $\cal A_3$ the set of all subsets of $\Ndb$ of cardinality at most $3$. For $\cal M \in \cal K(2^\Ndb)$, we denote $X_{\cal M}=T_{\cal M \cup \cal A_3}$. For any $\cal M \in \cal K(2^\Ndb)$, the basis $(u_n)$ of $X_{\cal M}$ is $1$-unconditional and shrinking. Note also that $\cal M \cup \cal A_3$ contains an infinite element if and only if $\cal M$ does. So $X_{\cal M}$ is reflexive if and only if $\cal M$ contains no infinite element and otherwise $X_{\cal M}$ is isomorphic to a $c_0$-sum of finite dimensional spaces. The space $Y_{\cal M}:=X_{\cal M}^*$ has a $1$-unconditional boundedly complete basis $(e_n)$ and so does $Y_{\cal M}^q$ (the $q$-convexification of $Y_{\cal M}$). We then set $Z_{\cal M}=(Y_{\cal M}^q)_*$. By abuse of notation, for any $\cal M \in \cal K(2^\Ndb)$, we shall denote $(e_n)$ the canonical basis of $Y_{\cal M}$ or $Y_{\cal M}^q$ and $(u_n)$ the canonical basis of $X_{\cal M}$ or $Z_{\cal M}$.
		
		If $\cal M\cap \cal I \neq \emptyset$, then  we apply property (4) of Proposition \ref{Kurka} and Lemma \ref{L1} to deduce that $Y_{\cal M}^q$ is isomorphic to an $\ell_q$-sum of finite dimensional spaces. Then $Z_{\cal M}$ is isomorphic to an $\ell_p$-sum of finite dimensional spaces and therefore is in $\textsf{T}_p$. On the other hand, if $\cal M\cap \cal I=\emptyset$, then $X_{\cal M}$ is reflexive and, by Lemma \ref{L3},  $Z_{\cal M} \notin \textsf{T}_p$.
		Our next goal is to show the existence of a Borel map $\Theta: \cal K(2^\Ndb)\to \textsf{SB}$ such that for all $\cal M \in \cal K(2^\Ndb)$, $\Theta(\cal M)$ is isometric to $Z_{\cal M}$. For this, we only have to slightly modify Kurka's argument. We shall first use the following.
		\begin{lemma}[Fact 3.4 in \cite{Kurka2018}]
			Let $\cal M_1,\cal M_2 \in \cal K(2^\Ndb)$ and $l\in \Ndb$ so that 
			$$\big\{A\cap \{1,\ldots,l\},\ A\in \cal M_1\big\}= \big\{A\cap \{1,\ldots,l\},\ A\in \cal M_2\big\}.$$
			Then $\|x\|_{T_{\cal M_1}}=\|x\|_{T_{\cal M_2}}$, for all $x\in \spa\{u_1,\ldots,u_l\}$.
		\end{lemma}
		We claim that the spaces $Z_{\cal M}$ satisfy the same property. Indeed, assume as above that 
		$$\big\{A\cap \{1,\ldots,l\},\ A\in \cal M_1\big\}= \big\{A\cap \{1,\ldots,l\},\ A\in \cal M_2\big\}.$$ Then,  
		$$\big\{A\cap \{1,\ldots,l\},\ A\in \cal M_1 \cup \cal A_3\big\}= \big\{A\cap \{1,\ldots,l\},\ A\in \cal M_2 \cup \cal A_3\big\}.$$
		So, $\|x\|_{X_{\cal M_1}}=\|x\|_{X_{\cal M_2}}$, for all $x\in \spa\{u_1,\ldots,u_l\}$.
		Thus, for all $y\in \spa\{e_1,\ldots,e_l\}$, $\|y\|_{Y_{\cal M_1}}=\|y\|_{Y_{\cal M_2}}$. Here we just use the definition of a dual norm and the fact that the basis $(u_n)$ is monotone. Clearly we then have that $\|y\|_{Y_{\cal M_1}^q}=\|y\|_{Y_{\cal M_2}^q}$ for all $y\in \spa\{e_1,\ldots,e_l\}$, and using Hahn-Banach Theorem and the fact that the basis $(e_n)$ is also monotone we get that  $\|x\|_{Z_{\cal M_1}}=\|x\|_{Z_{\cal M_2}}$, for all $x\in \spa\{u_1,\ldots,u_l\}$. Now we can deduce from \cite{Kurka2018} (see argument in the proof of Lemma 3.8) that for every $x\in c_{00}$, the map $\cal M \mapsto \|\sum_{i=1}^\infty x_iu_i\|_{Z_{\cal M}}$ is continuous from $\cal K(2^\Ndb)$ to $\Rdb$. It finally follows from Lemma 2.4 in \cite{Kurka2018} that there exists a Borel map $\Theta: \cal K(2^\Ndb)\to \textsf{SB}$ such that for all $\cal M \in \cal K(2^\Ndb)$, $\Theta(\cal M)$ is isometric to $Z_{\cal M}$.
		
		By Hurewicz's Theorem, $C=\{\cal M\in \cal K(2^\Ndb), \cal M \cap \cal I\neq \emptyset\}$ is $\Sigma_1^1$-complete in $\cal K(2^\Ndb)$ and we have that $\Theta(\cal M) \in \textsf{T}_p$ if and only if $\cal M \in C$. This is known to imply that $\textsf{T}_p \cap \textsf{{Sep}}$ is $\Sigma_1^1$-hard (see Lemma 2.1 in \cite{Kurka2018}) and therefore $\Sigma_1^1$-complete, as we already know that it is $\Sigma_1^1$.
		\end{proof}
		
We conclude this paper with a remark that has been kindly suggested to us by the referee. Recall first that it follows easily from Kwapien's theorem that the isomorphism class of $\ell_2$ is Borel. In \cite{Godefroy2017}, G. Godefroy proved that, for $1<p<\infty$, the isomorphism class of $\ell_p$ is also Borel. On the other hand, O. Kurka showed in \cite{Kurka2019} that the isomorphism class of $c_0$ is not Borel. The following statement is due to O. Kurka \cite{Kurka2018} (Remarks 3.10 $(ii)$ and $(vii)$) for $p=1$ and the same proof combined with our construction can be applied to obtain the following.

\begin{proposition}\label{referee} 
Let $p\in (1,\infty)$ and $\textsf{\emph{J}}=(J_n)_{n=1}^\infty$ be a Banach-Mazur dense sequence of finite dimensional spaces. Then, the isomorphism class of $\big(\sum_{n=1}^\infty J_n)_{\ell_p}$ is not Borel.
\end{proposition}

\begin{proof} Recall that for $\cal M \in \cal K(2^{\Ndb})$, $Z_{\cal M}$ is isomorphic to an $\ell_p$-sum of finite dimensional spaces if $\cal M$ contains an infinite subset of $\Ndb$ and not in $\textsf{T}_p$ otherwise. It follows easily that  $Z_{\cal M}$ is isomorphic to $\big(\sum_{n=1}^\infty J_n)_{\ell_p}$ if and only if $\cal M$ contains an infinite subset of $\Ndb$. The conclusion then follows from the previous arguments.
\end{proof}

\noindent{\bf Acknowledgements.} We are very grateful to the referee for her/his suggestions that helped improve the presentation of this paper and for indicating to us Proposition \ref{referee}.

	\bibliographystyle{amsplain}

\begin{bibsection}
		\begin{biblist}
			
			
			
			\bib{ArgyrosDeliyanni1997}{article}{
				author = {Argyros, S. A.},
				author={Deliyanni, I.},
				title = {{Examples of asymptotic \(\ell_{1}\) Banach spaces}},
				journal = {{Trans. Am. Math. Soc.}},
				volume = {349},
				number = {3},
				pages = {973--995},
				date = {1997},
			}
			
			\bib{Bossard1994}{article}{
				author={Bossard, B.},
				title={Théorie descriptive des ensembles en géométrie des espaces de Banach},
				journal={Thèse de doctorat de Mathématiques de l'Université Paris VI},
				volume={},
				date={1994},
				number={},
				pages={}
			}
			
			\bib{Bossard2002}{article}{
				author={Bossard, B.},
				title={A coding of separable Banach space. Analytic and coanalytic families of Banach spaces},
				journal={Fundam. Math.},
				volume={172},
				date={2002},
				number={},
				pages={117--151}
			}
			
			
			
			
			\bib{CauseyPositivity2018}{article}{
				author={Causey, R. M.},
				title={Power type asymptotically uniformly smooth and asymptotically
					uniformly flat norms},
				journal={Positivity},
				volume={22},
				date={2018},
				number={5},
				pages={1197--1221}
			}
			
			
			
			
			
			
			
			
			\bib{CauseyNavoyan}{article}{
				author={Causey, R. M.},
				author={Navoyan, K. V.},
				title={Factorization of Asplund operators},
				journal={J. Math. Anal. Appl.},
				volume={479},
				date={2019},
				number={1},
				pages={1324--1354}
			}
			
			
			
			
			\bib{DKLR2017}{article}{
				author={Dilworth, S. J.},
				author={Kutzarova, D.},
				author={Lancien, G.},
				author={Randrianarivony, N. L.},
				title={Equivalent norms with the property $(\beta)$ of Rolewicz},
				journal={Rev. R. Acad. Cienc. Exactas F\'{\i}s. Nat. Ser. A Mat. RACSAM},
				volume={111},
				date={2017},
				number={1},
				pages={101--113}
			}
			
			
			\bib{FOSZ}{article}{
				author={Freeman, D.},
				author={Odell, E.},
				author={Schlumprecht, Th.},
				author={Zs\'ak, A.},
				title={Banach spaces of bounded Szlenk index. II},
				journal={Fundam. Math.},
				volume={205},
				date={2009},
				number={2},
				pages={162--177}
			}
			
			\bib{Godefroy2017}{article}{
				author={Godefroy, G.},
				title={The isomorphism classes of $l_p$ are Borel},
				journal={Houston J. Math.},
				volume={43},
				date={2017},
				number={3},
				pages={947--951}
			}
			
			
			
			\bib{GKL}{article}{
				author={Godefroy, G.},
				author={Kalton, N.},
				author={Lancien, G.},
				title={Subspaces of \(c_0 (\mathbb{N})\) and Lipschitz isomorphisms},
				journal={Geom. Funct. Anal.},
				volume={10},
				date={2000},
				number={4},
				pages={798--820}
			}
		
		\bib{GSR}{article}{
			author={Godefroy, G.},
			author={Saint-Raymond, J.},
			title={Descriptive complexity of some isomorphism classes of {Banach} spaces},
			journal={J. Funct. Anal.},
			volume={275},
			date={2018},
			number={4},
			pages={1008--1022}
		}
		
		
		
		
		
			
			\bib{JLPS}{article}{
				author={Johnson, W. B.},
				author={Lindenstrauss, J.},
				author={Preiss, D.},
				author={Schechtman, G.},
				title={Almost Fr\'echet differentiability of Lipschitz mappings between infinite-dimensional Banach spaces},
				journal={Proc. Lond. Math. Soc. (3)},
				volume={84},
				date={2002},
				number={3},
				pages={711--746}
			}
			
			\bib{JohnsonZippin}{article}{
				author={Johnson, W.B.},
				author={Zippin, M.},
				title={Subspaces and quotient spaces of \((\sum G_n)_{l_p}\) and \((\sum G_n)_{c_0}\)},
				journal={Isr. J. Math.},
				volume={17},
				date={1974},
				number={},
				pages={50--55}
			}
			
			
			
			\bib{KW}{article}{
				author={Kalton, N.},
				author={Werner, D.},
				title={Property (M), M-ideals and almost isometric structure},
				journal={J. Reine Angew. Math.},
				volume={461},
				date={1995},
				number={},
				pages={137--178}
			}
			
			
			
			\bib{Kechris}{book}{
				author = {Kechris, A.},
				title = {Classical descriptive set theory},
				journal = {Grad. Texts Math.},
				volume = {156},
				number = {},
				pages = {},
				date = {1995},
				Publisher = {Springer-Verlag, Berlin},
			}
			
			
			
			
			\bib{LindenstraussTzafriri}{book}{
				author = {Lindenstrauss, J.},
				author = {Tzafriri, L.},
				title = {Classical Banach spaces I. Sequence spaces},
				journal = {},
				volume = {92},
				number = {},
				pages = {},
				date = {1977},
				Publisher = {Springer-Verlag, Berlin},
			}
			
			\bib{Kurka2018}{article}{
				Author = {Kurka, O.},
				title = {Tsirelson-like spaces and complexity of classes of Banach spaces},
				journal = {Rev. R. Acad. Cienc. Exactas F\'{\i}s. Nat., Ser. A Mat., RACSAM},
				volume = {112},
				number = {4},
				pages = {1101--1123},
				date = {2018}
			}
			
			
\bib{Kurka2019}{article}{
				Author = {Kurka, O.},
				title = {The isomorphism class of $c_0$ is not Borel},
				journal = {Israel J. Math.},
				volume = {231},
				number = {1},
				pages = {243--268},
				date = {2019}
			}			
			
			
			\bib{OS}{article}{
				Author = {Odell, E.},
				Author = {Schlumprecht, Th.},
				title = {Embeddings into Banach spaces with finite dimensional decompositions},
				journal = {Rev. R. Acad. Cienc. Exactas F\'{\i}s. Nat., Ser. A Mat., RACSAM},
				volume = {100},
				number = {},
				pages = {295--323},
				date = {2006}
			}
			
			
			\bib{Prus}{article}{
				Author = {Prus, S.},
				title = {Nearly uniformly smooth Banach Spaces},
				journal = {Boll. Unione Mat. Ital., VII. Ser., B},
				volume = {3},
				number = {3},
				pages = {507--521},
				date = {1989}
			}
			
			
			\bib{Schechtman1975}{article}{
				Author = {Schechtman, G.},
				Title = {{On Pelczynski's paper 'Universal bases'}},
				journal = {{Isr. J. Math.}},
				Volume = {22},
				Pages = {181--184},
				Year = {1975},
			}
			
			
		\end{biblist}
		
	\end{bibsection}

\end{document}